\documentclass[12pt]{article}

\usepackage{amsmath,amssymb,amsthm}
\usepackage[margin=1in]{geometry}
\usepackage{xcolor}
\usepackage{comment}

\usepackage{tikz}
\usetikzlibrary{arrows.meta,positioning,calc}

\usepackage[
  colorlinks=true,
  citecolor=red,
  linkcolor=black
]{hyperref}

\newtheorem{theorem}{Theorem}
\newtheorem{definition}{Definition}
\newtheorem{remark}{Remark}
\newtheorem{lemma}{Lemma}
\newtheorem{proposition}{Proposition}
\newtheorem{corollary}{Corollary}
\newtheorem{example}{Example}

\begin{document}

\title{The Fourier Ratio: A Unifying Measure of Complexity for Recovery, Localization, and Learning}

\author{W. Burstein, A. Iosevich, and H. Nathan}
\date{\today}
\maketitle

\begin{abstract}
We introduce a generalized Fourier ratio, the \(\ell^1/\ell^2\) norm ratio of coefficients in an \emph{arbitrary} orthonormal system, as a single, basis-invariant measure of \emph{effective dimension} that governs fundamental limits across signal recovery, localization, and learning. First, we prove that functions with small Fourier ratio can be stably recovered from random missing samples via \(\ell^1\) minimization, extending and clarifying compressed sensing guarantees for general bounded orthonormal systems. Second, we establish a sharp \emph{localization obstruction}: any attempt to localize recovery to subslices of a product space necessarily inflates the Fourier ratio by a factor scaling with the square root of the slice count, demonstrating that global complexity cannot be distributed locally. Finally, we show that the same parameter controls key complexity-theoretic measures: it provides explicit upper bounds on Kolmogorov rate-distortion description length and on the statistical query (SQ) dimension of the associated function class. These results unify analytic, algorithmic, and learning-theoretic constraints under a single complexity parameter, revealing the Fourier ratio as a fundamental invariant in information-theoretic signal processing.
\end{abstract}

\tableofcontents

\section{Background and motivation}
\label{sec:background}

Let $D$ be a finite set with $|D|=M$, and identify ${\mathbb C}^D$ with the space of
functions $f:D\to{\mathbb C}$. We use the $\ell^2$ inner product
$$
\langle f,g\rangle = \sum_{x\in D} f(x)\overline{g(x)}.
$$
With this convention, $\|f\|_2^2=\sum_{x\in D}|f(x)|^2$.

The classical uncertainty principle of Donoho and Stark gives recovery guarantees when a signal is sparse in the ``time'' domain and incomplete information in the Fourier domain is received, see \cite{DS89}. A more recent line of work replaces hard support size by magnitude-sensitive quantities. Chief among these is the Fourier ratio
$$
FR(f)=\frac{\|c(f)\|_1}{\|c(f)\|_2},
$$
where $c(f)$ is an appropriate coefficient vector (Fourier coefficients, Gabor coefficients, or coefficients in another orthonormal expansion). This ratio behaves like the square root of an effective support size and governs approximation and recovery phenomena.

A conceptual point that motivates this paper is that most classical compressed sensing
results are organized around sparsity or approximate sparsity parameters, whereas many
natural signals are not sparse even approximately, but nevertheless exhibit strong
compressibility in a global sense. The Fourier ratio provides a single quantitative
parameter that captures this global compressibility and interfaces directly with $\ell^1$
recovery theory. One of the goals of this paper is to make precise

\noindent\textbf{Main contributions.} We isolate the Fourier ratio as a single, basis-compatible notion of effective dimension that simultaneously controls positive recovery guarantees and negative localization phenomena, and that also yields explicit complexity-theoretic consequences. More precisely: (i) we restate the global Fourier-ratio recovery theorem from \cite{FRdiscrete} in a standard bounded-orthonormal-system setting (Theorem~\ref{thm:FR-ONB}) so it applies cleanly to structured bases beyond the Fourier basis; (ii) we specialize this basis formulation to orthonormal Gabor bases on ${\mathbb Z}_N\times{\mathbb Z}_T$ (Corollary~\ref{cor:gabor-global}); (iii) we prove a sharp obstruction showing that slicing/localization necessarily inflates Fourier-ratio complexity by a factor comparable to the square root of the number of slices (Theorem~\ref{thm:obstruction} and Proposition~\ref{prop:localization-abelian}); and (iv) we show that the same parameter yields explicit upper bounds for algorithmic rate distortion description length and for the statistical query complexity of the associated function class (Theorem~\ref{thm:kolmogorov-FR-abelian} and Theorem~\ref{thm:gaborstatdim}), through explicit examples and theorems, showing that Fourier-ratio-based recovery can apply in regimes where sparsity-based formulations are awkward or misleading.

In parallel, there is a different and very successful use of Gabor ideas in discrete
signal recovery that proceeds by localizing the Fourier transform row by row and then
studying frequency erasures within each row. A representative result is Theorem 1.8 from \cite{GaborFrames}, which we record explicitly below since it will be referenced
in the narrative.

The main point of the present paper is that there are two logically distinct ways in which
Gabor structures enter signal recovery. The first is global and basis-driven: one expands a
signal in an orthonormal Gabor basis, measures complexity by the Fourier ratio of the full
Gabor coefficient vector, and then applies a global $\ell^1$ recovery theorem. The second is
local and architecture-driven: one transmits localized Fourier data row by row and studies
frequency erasures inside each local block. We record one theorem of each type and explain
how they fit together. The final section proves two basic results that clarify why the global
theorem cannot in general be applied after localization.

It is also useful to emphasize what this paper is not claiming. The global theorem that drives
our discussion is a standard compressive sensing mechanism, and the basis formulation we record
does not introduce a new method. The point is that, when the complexity hypothesis is phrased in
terms of a Fourier ratio, the basis formulation gives a clean and immediately usable statement for
orthonormal Gabor expansions that does not seem to be written explicitly in the time-frequency
recovery literature, which often emphasizes different loss models.

A central conceptual contribution of this paper is the identification of Fourier ratio as a
complexity parameter that simultaneously governs positive recovery results and negative
localization phenomena. While a bounded Fourier ratio guarantees stable global recovery
from random missing data, we show that this control necessarily degrades under localization,
even on products of finite abelian groups. This obstruction demonstrates that Fourier-ratio-based
recovery is intrinsically global and cannot, in general, be reduced to slice-by-slice or
row-wise arguments. We further show that the same parameter controls algorithmic description
length and statistical query dimension, revealing a common notion of effective dimension
across analytic, algorithmic, and learning-theoretic settings.

This paper is not about improving recovery rates. It is about identifying a single quantitative invariant that simultaneously controls recovery, obstructs localization, and bounds algorithmic and learning-theoretic complexity.

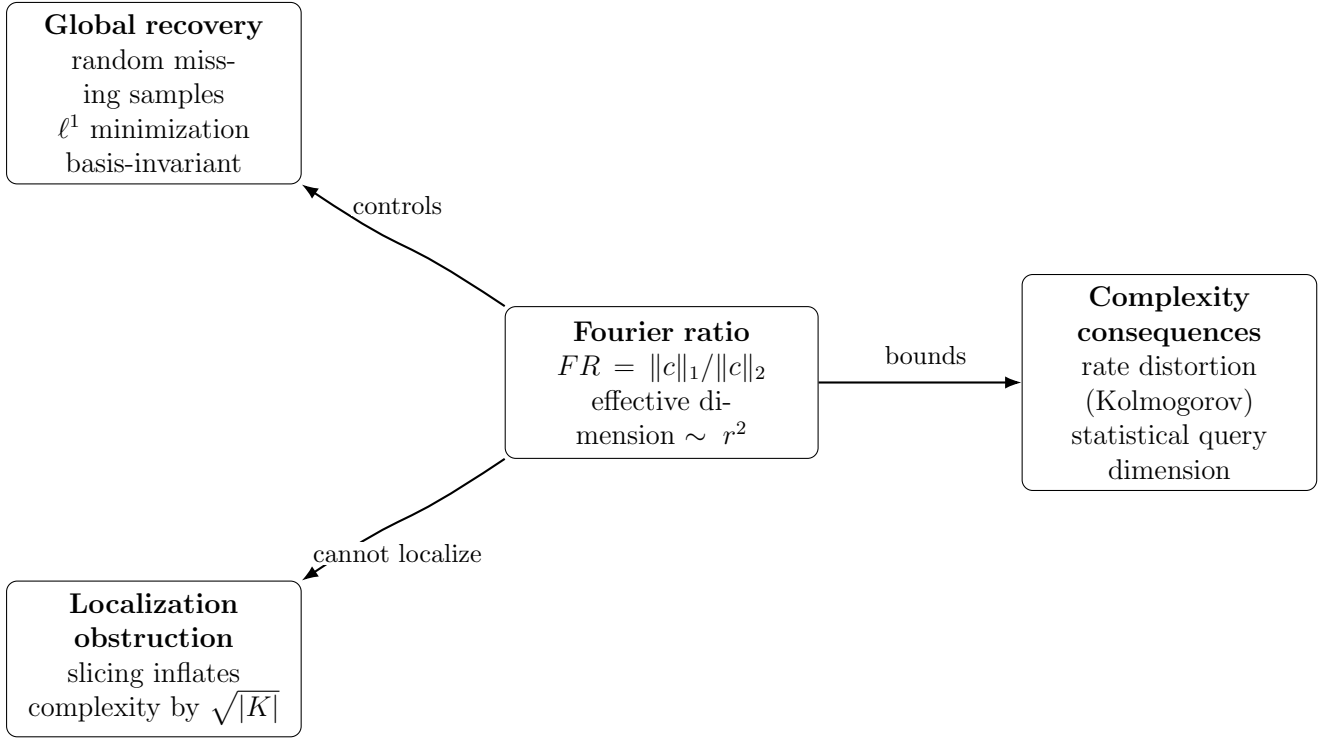
\begin{figure}[t]
\centering
\begin{tikzpicture}[
  scale=0.9,
  transform shape,
  >={Latex[length=2.2mm]},
  node distance=18mm and 30mm,
  box/.style={draw, rounded corners, align=center, inner sep=6pt, text width=42mm},
  small/.style={draw, rounded corners, align=center, inner sep=5pt, text width=40mm},
  lab/.style={align=center, inner sep=1pt, fill=white}
]

\node[box] (fr)
{\textbf{Fourier ratio}\\
$FR=\|c\|_1/\|c\|_2$\\
effective dimension $\sim r^2$};

\node[small, above left=of fr] (rec)
{\textbf{Global recovery}\\
random missing samples\\
$\ell^1$ minimization\\
basis-invariant};

\node[small, below left=of fr] (loc)
{\textbf{Localization obstruction}\\
slicing inflates\\
complexity by $\sqrt{|K|}$};

\node[small, right=of fr] (cmp)
{\textbf{Complexity consequences}\\
rate distortion (Kolmogorov)\\
statistical query dimension};

\draw[->, line width=0.8pt] (fr) -- (cmp);

\draw[->, line width=0.8pt]
  (fr.north west) .. controls +(-18mm,12mm) and +(18mm,-12mm) .. (rec.south east);

\draw[->, line width=0.8pt]
  (fr.south west) .. controls +(-18mm,-12mm) and +(18mm,12mm) .. (loc.north east);

\node[lab] at ($(fr)!0.52!(cmp) + (0,4mm)$) {\small bounds};
\node[lab] at ($(fr)!0.52!(rec) + (0,4mm)$) {\small controls};
\node[lab] at ($(fr)!0.52!(loc) + (0,-4mm)$) {\small cannot localize};

\end{tikzpicture}
\caption{Fourier ratio as a unifying complexity parameter controlling recovery, localization, and dimension.}
\end{figure}

\subsection{Road map}

Section~\ref{sec:background} records an earlier result in which one transmits row-wise Fourier (Gabor) data and studies
random frequency erasures. A key theme is that Fourier-ratio control is intrinsically global, section~\ref{sec:localization} shows that naive slice-by-slice localization can destroy the relevant complexity bounds even on product groups. Section~\ref{sec:background} also states, for ease of comparison, the global Fourier-ratio
recovery theorem from \cite{FRdiscrete}. Section~\ref{sec:basis} states the basis formulation of that theorem for
general bounded orthonormal expansions. Section~\ref{sec:gabor} specializes the basis theorem to two-dimensional
orthonormal Gabor bases. Section~\ref{sec:localization} discusses why localization can obstruct naive slice-by-slice use
of global Fourier-ratio bounds. Section~\ref{sec:complexity} discusses Kolmogorov-type description length and statistical query dimension in this setting. Section~\ref{sec:proofs} contains the proofs of the main results.

\subsection{Gabor recovery theorem with random frequency erasures}
\label{subsec:gabor_row_wise}

Let $f:{\mathbb Z}_N\times{\mathbb Z}_T\to{\mathbb C}$. For each $a\in{\mathbb Z}_T$,
the row-wise Gabor transform is the one-dimensional Fourier transform of the row
$t\mapsto f(t,a)$,
$$
Gf(m,a)=N^{-\frac{1}{2}}\sum_{t\in{\mathbb Z}_N} f(t,a)e^{-2\pi i mt/N}.
$$
\begin{theorem}[Row-wise Gabor transmission under binomial frequency losses {\cite[Theorem 1.8]{GaborFrames}}]
\label{thm:gabor-frames-18}
Suppose $f:{\mathbb Z}_N\times{\mathbb Z}_T\to{\mathbb C}$, and $T=T(N)$ satisfies
$T(N)=o(\sqrt{N}e^N)$. Define
$$
E_{\max}=\max_{a\in{\mathbb Z}_T} |\operatorname{supp}_t(f(t,a))|.
$$
Assume one transmits $Gf(m,a)$ for all $(m,a)\in{\mathbb Z}_N\times{\mathbb Z}_T$, and each
transmitted frequency is lost independently with fixed probability $\theta$, where
$0<\theta<\frac{1}{2E_{\max}}$. Let $M$ be the random set of missing frequencies and define
$$
M_{\max}=\max_{a\in{\mathbb Z}_T} |M\cap\{(m,a):m\in{\mathbb Z}_N\}|.
$$
Then, as $N\to\infty$,
$$
{\mathbb P}\Bigl(M_{\max}<\frac{N}{2E_{\max}}\Bigr)\to 1,
$$
which implies that the probability of unique recovery converges to $1$.
\end{theorem}

\begin{remark}
Theorem \ref{thm:gabor-frames-18} concerns random erasures of row-wise Fourier data and
a row-by-row recovery mechanism. In contrast, the global theorem discussed next concerns
recovery from random missing samples using an $\ell^1$ method controlled by a Fourier-ratio
bound on a single coefficient vector. The two viewpoints are compatible, but they address
different loss models and use different complexity parameters.
\end{remark}

\subsection{Why the Fourier ratio can be more natural than approximate sparsity}

Many compressed sensing theorems are stated in terms of sparsity or approximate sparsity. A typical quantitative formulation is that one can recover a vector up to an error controlled by the $\ell^1$ tail of its best $S$-term approximation. For example, Cand\`es, Romberg, and Tao prove the following (we quote it because it is a clean canonical statement of this type).

\begin{theorem}[Stable recovery for approximately sparse vectors {\cite[Theorem 2]{CRT05}}]
\label{thm:CRT-approx-sparse}
Assume the hypotheses of \cite[Theorem 1]{CRT05} (in particular a restricted isometry condition at order $4S$).
Let $x_0\in{\mathbb R}^n$ be arbitrary and let $x_{0,S}$ be the vector obtained by keeping the $S$ largest
entries of $x_0$ in magnitude and setting the rest to $0$. Suppose $y=Ax_0+e$ with $\|e\|_2\le \epsilon$,
and let $x^\ast$ be the solution of
$$
\min_x \|x\|_1
\quad\text{subject to}\quad
\|Ax-y\|_2\le \epsilon.
$$
Then
$$
\|x^\ast-x_0\|_2 \le C_{1,S}\epsilon + C_{2,S}\frac{\|x_0-x_{0,S}\|_1}{\sqrt{S}}.
$$
\end{theorem}

The restricted isometry condition referred in Theorem \ref{thm:CRT-approx-sparse} essentially requires any selection of $S$ columns from $A$ act approximately like an orthonormal basis. Theorem \ref{thm:CRT-approx-sparse} is extremely useful, but its hypothesis is expressed through a parameter
$S$ that is not always the right way to measure complexity in applications. In many problems one does not
expect exact sparsity, and even the best $S$-term approximation error $\|x_0-x_{0,S}\|_1$ can decay slowly
with $S$. In addition, determining if a matrix satisfies the restricted isometry property is NP-hard (see \cite{Tillman2014}).

The Fourier ratio captures a different notion of compressibility. It measures how concentrated the coefficient
vector is in the sense of $\ell^1/\ell^2$, without requiring an explicit choice of $S$. The next example shows
concretely why this can be preferable. In this example, the coefficient sequence is not approximately sparse in any meaningful sense: there is no threshold for which a small subset captures most of the $\ell^2$ norm. Nevertheless, the Fourier ratio remains small, and the recovery theorem applies. In short, approximate sparsity fails, but the Fourier ratio remains small.

\begin{example}[A simple coefficient model where Fourier ratio is informative]
\label{ex:FR-vs-approx-sparse}
Fix $M\ge 3$ and define a coefficient vector $c\in{\mathbb C}^M$ by
$$
c_j=\frac{1}{j},
\qquad
j=1,2,\dots,M.
$$
This vector is not sparse and does not have compact support. Nevertheless, it is highly compressible in the
sense that many coefficients are small and decay regularly.

We first compute its Fourier-ratio-type complexity.
We have
$$
\|c\|_1=\sum_{j=1}^M \frac{1}{j}
\ge \log M,
$$
and also
$$
\|c\|_2^2=\sum_{j=1}^M \frac{1}{j^2}
\le \sum_{j=1}^\infty \frac{1}{j^2}
=
\frac{\pi^2}{6}.
$$
Therefore
$$
\frac{\|c\|_1}{\|c\|_2}
\ge
\frac{\log M}{\pi/\sqrt{6}}
=
\frac{\sqrt{6}}{\pi}\log M.
$$
In particular, the Fourier-ratio complexity of this vector grows only like $\log M$, even though the vector has
$M$ nonzero entries.

Now compare this with the approximate sparsity term in Theorem \ref{thm:CRT-approx-sparse}.
For any $S<M$, the best $S$-term approximation is obtained by keeping $c_1,\dots,c_S$, so
$$
\|c-c_S\|_1=\sum_{j=S+1}^M \frac{1}{j}
\ge \log\Bigl(\frac{M}{S}\Bigr).
$$
Thus the deterministic error term in Theorem \ref{thm:CRT-approx-sparse} has size at least
$$
\frac{\|c-c_S\|_1}{\sqrt{S}}
\ge
\frac{1}{\sqrt{S}}\log\Bigl(\frac{M}{S}\Bigr).
$$
To make this quantity smaller than a prescribed accuracy parameter $\eta$, one must take $S$ large enough that
$$
\sqrt{S}\ge \frac{1}{\eta}\log\Bigl(\frac{M}{S}\Bigr),
$$
which forces $S$ to scale at least on the order of $(\log M)^2/\eta^2$ (up to lower-order logarithmic effects).

By contrast, the Fourier ratio $\|c\|_1/\|c\|_2$ immediately reports the correct logarithmic size of the
complexity of this vector. This is exactly the reason why our main recovery theorem is stated in terms of
Fourier ratio rather than in terms of an explicit approximate sparsity level.
\end{example}

The purpose of this paper is not to replace sparsity-based formulations of compressed
sensing, but rather to complement them. In many natural situations, signals are not
sparse or even approximately sparse in a meaningful sense, yet they exhibit strong
global compressibility that is well captured by the Fourier ratio. The results in this
paper show that Fourier-ratio-based recovery fits naturally into the existing $\ell^1$
framework and leads to clean recovery guarantees in regimes where sparsity-based
descriptions are awkward or misleading.

Unlike approximate sparsity, which depends on thresholding and is sensitive to basis choice, the Fourier ratio provides a basis-compatible, quantitative notion of effective dimension that continues to govern recovery even when no sparse model is available.

\subsection{The global Fourier-ratio theorem and why it is genuinely global}

A central result in \cite{FRdiscrete} is a global recovery theorem (Theorem 1.21 there)
showing that if a signal has small Fourier ratio, then it can be stably recovered from
randomly missing samples by $\ell^1$ minimization of its Fourier transform. Since we will
compare later arguments to this theorem we state it explicitly.

\begin{theorem}[Global Fourier-ratio recovery {\cite[Theorem 1.21]{FRdiscrete}}]
\label{thm:FR121}
Let $G$ be a finite abelian group and let $f:G\to{\mathbb C}$ be nonzero. Fix $\epsilon\in(0,1)$.
Assume
$$
FR(f)=\frac{\|\widehat f\|_1}{\|\widehat f\|_2}\le r,
$$
where $\widehat f$ is the Fourier transform on $G$ normalized so that Parseval holds. Let $X\subset G$ be a random subset obtained by keeping each point independently with probability $p$.
There exists absolute constants $c, C>0$ such that if
$$
p|G|\ge C\frac{r^2}{\epsilon^2}\log\Bigl(\frac{r}{\epsilon}\Bigr)^2\log |G|,
$$
then with probability at least $1-\exp(-c p|G|)$ the solution $f^*$ to
$$
\min_g \|\widehat g\|_1
\quad\text{subject to}\quad
\|g-f\|_{L^2(X)}\le \epsilon\|f\|_2
$$
satisfies
$$
\|f^*-f\|_2\le 11.47\,\epsilon\|f\|_2.
$$
\end{theorem}

\begin{remark}
Theorem \ref{thm:FR121} is genuinely global. Both the hypothesis and the conclusion depend only
on the Fourier ratio of the full signal $f$. The reason we emphasize this point is that it becomes
tempting, in two-dimensional problems, to localize first and then apply the global theorem to each
slice. Section~\ref{sec:localization} explains why this can fail without additional hypotheses.
\end{remark}

\section{A basis recovery theorem}
\label{sec:basis}

Theorem \ref{thm:FR121} is written in the Fourier basis, but the argument is built from two
ingredients that have standard analogues for general bounded orthonormal systems. The first is
the ``soft sparsity'' mechanism coming from a bound of the form $\|c\|_1\le r\|c\|_2$, which forces
a coefficient vector to have controlled effective support. The second is a restricted isometry type
estimate for randomly sampled measurement operators associated to bounded orthonormal systems.
Since both ingredients are present with various bases, Theorem \ref{thm:FR121} can be restated in a basis
form that is standard in compressive sensing.

The purpose of this section is therefore not to introduce a new recovery mechanism, but to isolate the precise structural assumptions under which the Fourier-ratio argument from Theorem \ref{thm:FR121} continues to apply.

\subsection{Coefficient vectors and the Fourier ratio}

\begin{definition}[Fourier ratio in an orthonormal basis]
Let ${\mathcal B}=\{\phi_j\}_{j\in D}$ be an orthonormal basis of ${\mathbb C}^D$.
For $f\in{\mathbb C}^D$, define its coefficient vector
$$
c_{\mathcal B}(f)=\{\langle f,\phi_j\rangle\}_{j\in D}.
$$
The Fourier ratio of $f$ with respect to ${\mathcal B}$ is
$$
FR_{\mathcal B}(f)=\frac{\|c_{\mathcal B}(f)\|_1}{\|c_{\mathcal B}(f)\|_2}.
$$
\end{definition}

\subsection{A soft sparsification lemma in Fourier-ratio language}

The next lemma is the quantitative point at which a Fourier-ratio bound is converted into a
statement of effective sparsity. It is standard, but we record it explicitly since it is
the bridge between Fourier-ratio hypotheses and the usual sparse approximation statements
used in compressive sensing.

\begin{lemma}[Small Fourier ratio implies approximate sparsity]
\label{lem:FR-sparse}
Let $c\in{\mathbb C}^M$ be nonzero and assume $\|c\|_1\le r\|c\|_2$. Fix $\eta\in(0,1)$ and let
$S$ be the set of indices of the $s$ largest coordinates of $c$ in magnitude, where
$$
s=\left\lceil \frac{r^2}{\eta^2}\right\rceil.
$$
Let $c_S$ be the truncation of $c$ to $S$. Then
$$
\|c-c_S\|_2\le \eta\|c\|_2.
$$
\end{lemma}

\begin{proof}
Write the coordinates of $c$ in nonincreasing order of magnitude,
$$
|c_1|\ge |c_2|\ge \cdots \ge |c_M|.
$$
Since $\sum_{j=1}^M |c_j|=\|c\|_1\le r\|c\|_2$, we have for every $j$ that
$$
j|c_j|\le \sum_{k=1}^j |c_k|\le \|c\|_1\le r\|c\|_2,
$$
so
$$
|c_j|\le \frac{r}{j}\|c\|_2.
$$
Therefore
$$
\|c-c_S\|_2^2=\sum_{j>s}|c_j|^2\le r^2\|c\|_2^2\sum_{j>s}\frac{1}{j^2}\le r^2\|c\|_2^2\frac{1}{s}.
$$
With $s\ge \frac{r^2}{\eta^2}$, this gives $\|c-c_S\|_2\le \eta\|c\|_2$.
\end{proof}

\begin{remark}
Lemma \ref{lem:FR-sparse} is not the sharpest formulation. The proof of \cite[Theorem 1.21]{FRdiscrete}
uses a refined bound that contains the factor $\log\Bigl(\frac{r}{\epsilon}\Bigr)^2$. We will utilize that refinement
later when proving Theorem \ref{thm:FR-ONB}, but the present lemma captures the basic idea.
\end{remark}

 \subsection{A basis version of the signal recovery theorem in \cite{FRdiscrete}}

We now state the basis theorem in a standard bounded orthonormal system setting. The additional
hypothesis
$$
\max_{x\in D}\max_{j\in D} |\phi_j(x)|\le M^{-\frac{1}{2}}
$$
is the usual incoherence condition needed for random subsampling to satisfy a restricted isometry
property, and it is exactly the mechanism that is automatic for the Fourier basis.

\begin{theorem}[Fourier-ratio recovery for bounded orthonormal expansions]
\label{thm:FR-ONB}
Let $D$ be a finite set with $|D|=M$. Let ${\mathcal B}=\{\phi_j\}_{j\in D}$ be an orthogonal
basis of ${\mathbb C}^D$ such that
$$
\max_{x\in D}\max_{j\in D} |\phi_j(x)| \le M^{-\frac{1}{2}}.
$$
Fix $\epsilon\in(0,1)$ and let $f\in{\mathbb C}^D$ be nonzero with
$$
FR_{\mathcal B}(f)\le r.
$$
Let $X\subset D$ be a random subset obtained by keeping each point independently with
probability $p$.

There exists absolute constants $c, C>0$ such that if
$$
pM \ge C \frac{r^2}{\epsilon^2}\log\Bigl(\frac{r}{\epsilon}\Bigr)^2 \log M,
$$
then with probability at least
$$
1-\exp(-c pM),
$$
the solution $f^*$ to the convex program
$$
\min_g \|c_{\mathcal B}(g)\|_1
\quad\text{subject to}\quad
\|g-f\|_{\ell^2(X)}\le \epsilon\|f\|_2
$$
satisfies
$$
\|f^*-f\|_2 \le 11.47\,\epsilon\|f\|_2.
$$
\end{theorem}

\begin{remark}[Scope of the basis hypothesis]
The boundedness assumption
$$
\max_{x\in D}\max_{j\in D}|\phi_j(x)|\le M^{-\frac{1}{2}}
$$
is standard in compressive sensing and should be viewed as a property of the
interaction between the sampling basis and the sparsifying basis, rather
than as a property of sparsity itself. It is automatic for the Fourier basis on
finite abelian groups and holds for many structured orthonormal systems arising in
time frequency analysis, including the orthonormal Gabor bases considered below.

By contrast, neither the row-wise Gabor transform (Subsection~\ref{subsec:gabor_row_wise}) nor commonly used orthonormal wavelet bases are typically well localized
in space and therefore fail to satisfy such a uniform bound when sampling is
performed in the standard basis. This does not reflect a deficiency of either, but rather the fact that random pointwise subsampling is
inherently coherent with spatially localized bases. In particular, the present
theorem does not assert that Fourier-ratio recovery holds for all orthonormal
systems, but instead isolates the precise mechanism: recovery is controlled by the
Fourier ratio of the coefficient vector together with an incoherent sampling
model.

This distinction helps explain why results of the present type were not previously
formulated explicitly. Much of the compressed sensing literature is organized
around sparsity models tied to specific bases, whereas the Fourier ratio provides
a basis-independent complexity parameter whose effectiveness depends on the
sampling architecture.
\end{remark}

\begin{remark}
Theorem \ref{thm:FR-ONB} is the basis analogue of Theorem 1.21 in \cite{FRdiscrete}. The constant $11.47$
is the same numerical constant recorded in \cite{FRdiscrete}. One may replace $11.47$ by a generic
absolute constant if one prefers to avoid tracking the final numerical step, but we keep the same
number to emphasize that the proof is the same proof, written in basis language.
\end{remark}

\subsection{A quick coherence check for constant-modulus bases}

In many discrete settings, one encounters orthonormal bases whose entries all have the same modulus
in the sampling domain. This immediately implies the boundedness condition in Theorem \ref{thm:FR-ONB}.

\begin{lemma}[Constant-modulus bases are bounded]
\label{lem:constant-modulus}
Let $D$ be a finite set with $|D|=M$ and let ${\mathcal B}=\{\phi_j\}_{j\in D}$ be an orthonormal
basis of ${\mathbb C}^D$. Assume that for every $j\in D$ and every $x\in D$,
$$
|\phi_j(x)|=M^{-\frac{1}{2}}.
$$
Then ${\mathcal B}$ satisfies the boundedness condition
$$
\max_{x\in D}\max_{j\in D}|\phi_j(x)|\le M^{-\frac{1}{2}}.
$$
\end{lemma}

\begin{proof}
This is immediate from the hypothesis.
\end{proof}

\begin{remark}
The standard Fourier basis on a finite abelian group is of this form. Many other structured
orthonormal systems also have constant modulus, including the Walsh-Hadamard system on
${\mathbb Z}_2^n$.
\end{remark}

\section{Application to two-dimensional Gabor bases}
\label{sec:gabor}

In this section, we specialize to $D={\mathbb Z}_N\times{\mathbb Z}_T$, so $M=NT$, and consider an
orthonormal Gabor basis
$$
{\mathcal G}=\{g_{m,a}\}_{(m,a)\in{\mathbb Z}_N\times{\mathbb Z}_T}
$$
for ${\mathbb C}^{{\mathbb Z}_N\times{\mathbb Z}_T}$, as in \cite{GaborFrames}. The coefficient
vector is
$$
c_{\mathcal G}(f)=\{\langle f,g_{m,a}\rangle\}_{m,a}.
$$

\begin{corollary}[Corollary to Theorem \ref{thm:FR-ONB}: Global recovery using Gabor coefficients]
\label{cor:gabor-global}
Let $f:{\mathbb Z}_N\times{\mathbb Z}_T\to{\mathbb C}$ and suppose the orthonormal Gabor basis
${\mathcal G}$ satisfies
$$
\max_{(t,a)\in{\mathbb Z}_N\times{\mathbb Z}_T}\max_{(m,n)\in{\mathbb Z}_N\times{\mathbb Z}_T}
|g_{m,n}(t,a)| \le (NT)^{-\frac{1}{2}}.
$$
Define
$$
FR_{\mathcal G}(f)=\frac{\|c_{\mathcal G}(f)\|_1}{\|c_{\mathcal G}(f)\|_2}.
$$
Assume $FR_{\mathcal G}(f)\le r$. Let $X\subset{\mathbb Z}_N\times{\mathbb Z}_T$ be a random
subset obtained by independent sampling with probability $p$. If
$$
pNT \ge C \frac{r^2}{\epsilon^2}\log\Bigl(\frac{r}{\epsilon}\Bigr)^2 \log(NT),
$$
then with probability at least $1-\exp(-c pNT)$ the signal $f$ is stably recoverable from
its samples on $X$ by minimizing $\|c_{\mathcal G}(g)\|_1$ subject to the $\ell^2(X)$ fidelity
constraint, and the error satisfies
$$
\|f^*-f\|_2 \le 11.47\,\epsilon\|f\|_2.
$$
\end{corollary}

\begin{remark}
Corollary \ref{cor:gabor-global} explains precisely in what sense the global theorem from
\cite{FRdiscrete} applies to the Gabor setting. One views the full Gabor coefficient vector as
a single orthonormal expansion, and then Theorem \ref{thm:FR-ONB} applies directly.
This is conceptually different from Theorem \ref{thm:gabor-frames-18}, which studies
frequency erasures in the row-wise Fourier data and uses a maximal row-support criterion.
\end{remark}

\subsection{Other natural orthonormal systems}

Theorem \ref{thm:FR-ONB} is formulated for bounded orthonormal bases. Besides Fourier and orthonormal
Gabor bases, there are several other standard orthonormal systems where the same conclusion applies.

One family consists of constant-modulus bases, to which Lemma \ref{lem:constant-modulus} applies directly.
Examples include the Walsh-Hadamard basis on ${\mathbb Z}_2^n$, and more generally Fourier-type character
bases on finite abelian groups.

A second family consists of wavelet-type orthonormal bases on discrete grids, such as the Haar basis on
dyadic grids. These bases are not constant modulus, but they are still bounded in the sense required by
Theorem \ref{thm:FR-ONB} after normalization, and the same random sampling mechanism applies.

A third family arises from orthonormal polynomial systems on finite sets equipped with a uniform measure,
where boundedness may hold after suitable normalization. In such examples, verifying the bound
$\max_{x,j}|\phi_j(x)|\le M^{-\frac{1}{2}}$ is a concrete, checkable step, and once it is verified the
Fourier-ratio recovery conclusion follows exactly as stated.

\section{Why localization creates difficulties}
\label{sec:localization}

The previous section shows that there is no obstruction to applying the Fourier-ratio recovery
theorem to Gabor coefficients when they are treated globally. Difficulties arise only when one
attempts to apply the theorem after localization, for example by slicing a two-dimensional signal
into rows and treating each row independently.

We record a simple universal obstruction in the Fourier case. For $f:{\mathbb Z}_N\times{\mathbb Z}_T
\to{\mathbb C}$, write $f_a(t)=f(t,a)$ for the $a$th row.

\begin{theorem}[Localization obstruction]
\label{thm:obstruction}
Let $f:{\mathbb Z}_N\times{\mathbb Z}_T\to{\mathbb C}$. Then
$$
\max_{a\in{\mathbb Z}_T} FR(f_a) \ge \frac{FR(f)}{\sqrt{T}},
$$
where $FR(f)$ is computed using the two-dimensional Fourier transform on
${\mathbb Z}_N\times{\mathbb Z}_T$ and $FR(f_a)$ is computed using the one-dimensional Fourier
transform on ${\mathbb Z}_N$.
\end{theorem}

\begin{remark}
Theorem \ref{thm:obstruction} is not a contradiction with the global results. It says that even
though the rows are orthogonal in $\ell^2$, the $\ell^1$ mass of Fourier coefficients will concentrate
heavily on a single row. In particular, one cannot expect uniform Fourier-ratio control on all
rows unless the global Fourier ratio is already extremely small.
\end{remark}

\begin{proposition}[Localization obstruction on finite abelian groups]
\label{prop:localization-abelian}
Let $G$ be a finite abelian group admitting a decomposition
$$
G = H \oplus K.
$$
Let $f : G \to {\mathbb C}$, and for each $k \in K$ define the slice
$$
f_k(h) = f(h,k), \quad h \in H.
$$
Let $FR_G(f)$ denote the Fourier ratio of $f$ computed using the Fourier transform on $G$,
and let $FR_H(f_k)$ denote the Fourier ratio of $f_k$ computed using the Fourier transform
on $H$.

Then one has the bound
$$
\max_{k \in K} FR_H(f_k) \ge \frac{FR_G(f)}{\sqrt{|K|}}.
$$
\end{proposition}

\begin{remark}
The proof of Proposition \ref{prop:localization-abelian} is identical to that of Theorem \ref{thm:obstruction}. It uses only Parseval's
identity and the Cauchy Schwarz inequality applied across the $K$ variable, and does not
depend on any arithmetic or geometric structure of the group beyond orthogonality of the
Fourier decomposition.

This formulation shows that the degradation of Fourier-ratio control under localization
is governed solely by the number of orthogonal components $|K|$. In particular, the factor
$\sqrt{T}$ appearing in Theorem \ref{thm:obstruction} reflects the size of the slicing index set rather than any
special feature of the time frequency setting.

From a complexity-theoretic point of view, the proposition indicates that localization
forces an intrinsic inflation of effective complexity by a factor comparable to $|K|$.
This observation will reappear in Section~\ref{sec:complexity}, where both Kolmogorov-type description length
bounds and statistical query dimension estimates scale naturally with the same parameter.
\end{remark}

\section{Measures of Complexity}
\label{sec:complexity}

\subsection{Connections with the theory of Kolmogorov complexity}

We briefly explain how a Fourier-ratio bound yields an explicit algorithmic description-length bound.
Throughout this subsection, $G$ denotes a finite abelian group equipped with the uniform probability
measure, and we regard $G$ and a fixed and uniform encoding of $\widehat G$ as background data. In particular,
given $\gamma\in \widehat G$, we assume there is a canonical way to represent $\gamma$ by a binary
string of length at most $C\log|G|$, where $C>0$ is an absolute constant.

\begin{definition}[\cite{LV19}] \label{def:kolmogorovcomplexity} Let \(U\) be a fixed universal prefix-free Turing machine.  
For a function \( f : \mathbb{Z}_N \to \mathbb{C} \), its Kolmogorov complexity \(K_U(f)\) is defined as
\[
K_U(f) \;=\; \min \{ \, |p| \;:\; U(p, x) = f(x) \ \text{for all } x \in \mathbb{Z}_N \,\},
\]
where \(|p|\) denotes the length of the binary program \(p\).
\end{definition} 

As discussed in \cite{FRdiscrete}, this definition is too sensitive to small perturbations for us to use. As such, we use the following definition.

\begin{definition}\cite{VV10}
Let \(U\) be a fixed universal prefix-free Turing machine.  For a function \( f : \mathbb{Z}_N \to \mathbb{C} \), its algorithmic rate distortion complexity of $f$ at accuracy
$\epsilon\in(0,1)$ by
$$
K_U(f,\epsilon)=\min\{K_U(s): s\ \text{is a binary program and}\ \|U(s)-f\|_2\le \epsilon\|f\|_2\},
$$
where $U(s)$ is interpreted as the function on $G$ output by running $U$ on input $s$ and then
decoding the output into an element of ${\mathbb C}^G$. If $U(s)$ does not halt or does not output a
function on $G$, we interpret $\|U(s)-f\|_2=\infty$.
\end{definition}

We normalize the Fourier transform on $G$ so that Parseval holds, so
$$
\|f\|_2=\|\widehat f\|_2.
$$
We write
$$
FR_G(f)=\frac{\|\widehat f\|_1}{\|\widehat f\|_2}.
$$

\begin{theorem}[Fourier ratio controls algorithmic rate distortion on finite abelian groups]
\label{thm:kolmogorov-FR-abelian}
Let $G$ be a finite abelian group and let $f:G\to{\mathbb C}$ be nonzero. Fix $\epsilon\in(0,1)$
and assume
$$
FR_G(f)\le r.
$$
There exist absolute constants $C_0,C_1>0$ and a constant $C_U>0$ depending only on the choice of
universal machine $U$ such that
$$
K_U(f,\epsilon)
\le
C_0\frac{r^2}{\epsilon^2}\log\Bigl(\frac{r}{\epsilon}\Bigr)^2\log|G|
+
C_1\frac{r^2}{\epsilon^2}\log\Bigl(\frac{r}{\epsilon}\Bigr)^3
+
C_U.
$$
In particular, for fixed $\epsilon$, the description length grows at most on the order of
$\frac{r^2}{\epsilon^2}\log(\frac{r}{\epsilon})^2\log|G|$.
\end{theorem}

\begin{remark}
Theorem \ref{thm:kolmogorov-FR-abelian} shows that a Fourier-ratio bound controls the algorithmic
description length of a signal on $G$ at fixed accuracy. In the setting of a decomposition
$G=H\oplus K$, the localization obstruction in Section~\ref{sec:localization} indicates that naive
slice-by-slice localization can inflate the effective Fourier-ratio complexity by a factor comparable
to $|K|$. This is consistent with the bound in Theorem \ref{thm:kolmogorov-FR-abelian}, where the
dominant term scales like $k\log|G|$ with $k$ proportional to the square of the relevant Fourier-ratio
parameter.
\end{remark}

\begin{corollary}[Algorithmic description length for the Gabor Fourier-ratio class]
\label{cor:kolmogorov-gabor-class}
Let $N,T\ge 1$ and let $G={\mathbb Z}_N\times{\mathbb Z}_T$. Fix an orthonormal Gabor basis
${\mathcal G}=\{g_{m,a}\}_{(m,a)\in{\mathbb Z}_N\times{\mathbb Z}_T}$ for ${\mathbb C}^G$, and for
$f:G\to{\mathbb C}$ define
$$
FR_{\mathcal G}(f)=\frac{\|c_{\mathcal G}(f)\|_1}{\|c_{\mathcal G}(f)\|_2}.
$$
Fix $\epsilon\in(0,1)$ and $r>0$. If $f:G\to{\mathbb C}$ is nonzero and satisfies
$$
FR_{\mathcal G}(f)\le r,
$$
then there exist absolute constants $C_0,C_1>0$ and a constant $C_U>0$ depending only on the
universal machine $U$ such that
$$
K_U(f,\epsilon)
\le
C_0\frac{r^2}{\epsilon^2}\log\Bigl(\frac{r}{\epsilon}\Bigr)^2\log(NT)
+
C_1\frac{r^2}{\epsilon^2}\log\Bigl(\frac{r}{\epsilon}\Bigr)^2\log\Bigl(\frac{1}{\epsilon}\Bigr)
+
C_U.
$$
In particular, for fixed $\epsilon$, every function in the class
$$
{\mathcal B}(r)=\{f:{\mathbb Z}_N\times{\mathbb Z}_T\to{\mathbb C}: FR_{\mathcal G}(f)\le r,\ f\ne 0\}
$$
admits an $\epsilon$-accurate description whose length grows at most on the order of
$\frac{r^2}{\epsilon^2}\log(\frac{r}{\epsilon})^2\log(NT)$.
\end{corollary}

\begin{proof}
Apply Theorem \ref{thm:kolmogorov-FR-abelian} to the finite abelian group
$G={\mathbb Z}_N\times{\mathbb Z}_T$ using the orthonormal basis ${\mathcal G}$ in place of the
Fourier basis. Since ${\mathcal G}$ is an orthonormal basis of ${\mathbb C}^G$, the same
sparsification and quantization argument applies verbatim to the coefficient vector
$c_{\mathcal G}(f)$, and $|G|=NT$.
\end{proof}

\subsection{Statistical Query Dimension}

\begin{definition}
For a concept class, $\mathcal{C} \subset \{-1,1\}^X$, and probability distribution, $\mathcal{D}$, on $X$, the statistical query
dimension of $\mathcal{C}$ with respect to $\mathcal{D}$ is the largest number $d$ such that $\mathcal{C}$ contains $d$
functions $f_1, f_2, . . . , f_d$ such that for all $i \not = j$,
$$
\left |\mathbb{E}_{x \sim \mathcal{D}}
\left
[f_i(x) \cdot f_j(x)
\right
] \right| \leq \frac{1}{d}.
$$
The Statistical Query Dimension of $\mathcal{C}$ is the maximum of the statistical query
dimension of $\mathcal{C}$ with respect to $\mathcal{D}$ over all $\mathcal{D}$.
\end{definition}

\begin{definition}
Let $G$ be a finite group with $M = |G|, $ and let $r > 0$. Then we define
$$
\mathcal{B}(r) = \{f:G \to \{-1, 1\} : FR_{\mathcal B}(f) < r\},
$$
where $FR_{\mathcal B}(f)$ is the Fourier with respect to the orthonormal basis $\{\phi_j\}_{j \in G}$.
\end{definition}

\begin{remark}
In the proof of Theorem~\ref{thm:gaborstatdim}, the constants are not optimized. One could perhaps get a tighter upper bound on the statistical query dimension of $\mathcal{B}(r)$ by optimizing the constants more carefully.
\end{remark}

\begin{theorem}
\label{thm:gaborstatdim}
The statistical query dimension of $\mathcal{B}(r)$ is bounded above by
$$
M^{16^2 M\tau^2 r^2} (16^3\tau^3 M\ r^2 + 1) (16\tau r \sqrt{M})^{16^2 M\tau^2 r^2}
$$
where
$$
\tau = \max_{x \in G} \max_{j \in G} |\phi_j(x)|.
$$
\end{theorem}

\begin{corollary}[A readable upper bound]
\label{cor:sq-readable}
There exists an absolute constant $C>0$ such that the statistical query dimension of $\mathcal{B}(r)$ satisfies
$$
\mathrm{SQdim}(\mathcal{B}(r)) \le \exp\Big(C\,M\tau^2 r^2\,\log\big(M\tau r\big)\Big),
$$
where $M=|G|$ and $\tau = \max_{x\in G}\max_{j\in G}|\phi_j(x)|$. In particular, for fixed $\tau$ and $r$ this grows at most exponentially in $M \log(M)$.
\end{corollary}

\section{Proofs of the main results}
\label{sec:proofs}

\subsection{Proof of Theorem \ref{thm:gabor-frames-18}}

Theorem \ref{thm:gabor-frames-18} is a previously known result and is stated here only for context.
We refer the reader to \cite[Section 1]{GaborFrames} for the full proof.

\subsection{Proof of Theorem \ref{thm:FR121}}

Theorem \ref{thm:FR121} is \cite[Theorem 1.21]{FRdiscrete}. We refer the reader to \cite[Section 7]{FRdiscrete}
for the complete proof, including the explicit numerical constant $11.47$.

\subsection{Proof of Theorem \ref{thm:FR-ONB}}

We give a complete proof in the same spirit as the proof of Theorem 1.21 in \cite{FRdiscrete}.
The only difference is that we quote standard compressive sensing facts for bounded orthonormal systems,
as in \cite{FoucartRauhut}, instead of quoting the corresponding special case for Fourier matrices.

Let ${\mathcal B}=\{\phi_j\}_{j\in D}$ be as in the statement and define the unitary map
$$
A:{\mathbb C}^D\to{\mathbb C}^D,
\qquad
Af=c_{\mathcal B}(f)=\{\langle f,\phi_j\rangle\}_{j\in D}.
$$
Since ${\mathcal B}$ is an orthonormal basis, $A$ is unitary, so $\|Af\|_2=\|f\|_2$ for every $f$.
The hypothesis $FR_{\mathcal B}(f)\le r$ is therefore
$$
\|Af\|_1\le r\|f\|_2.
$$

Let $X\subset D$ be the random sample set. Write $P_X:{\mathbb C}^D\to{\mathbb C}^D$ for the coordinate
projection that keeps entries on $X$ and sets all other entries to $0$. The constraint
$\|g-f\|_{\ell^2(X)}\le \epsilon\|f\|_2$ is the same as
$$
\|P_X(g-f)\|_2\le \epsilon\|f\|_2.
$$
The optimization problem in Theorem \ref{thm:FR-ONB} can thus be written as
$$
\min_g \|Ag\|_1
\quad\text{subject to}\quad
\|P_X(g-f)\|_2\le \epsilon\|f\|_2.
$$

The first step is to turn the $\ell^1$ control into a quantitative sparse approximation statement for
the coefficient vector $Af$. Write $c=Af\in{\mathbb C}^D$. Since $A$ is unitary, the desired approximation
in signal space is equivalent to an approximation of $c$ in coefficient space.

One can obtain a first approximation using Lemma \ref{lem:FR-sparse}. However, just as in the proof of \cite[Theorem 1.21]{FRdiscrete}, using Theorem 4.5 \cite{HavivRegev} this can be refined with the result that there exists a subset $S\subset D$ with
$$
|S|\le C_0\frac{r^2}{\epsilon^2}\log\Bigl(\frac{r}{\epsilon}\Bigr)^2 \log(M)
$$
such that, if $c_S$ is the truncation of $c$ to $S$, then
$$
\|c-c_S\|_2\le \epsilon\|c\|_2.
$$
Define the $|S|$-term approximation of $f$ by
$$
f_S=A^*c_S.
$$
Since $A$ is unitary, $\|f-f_S\|_2=\|c-c_S\|_2\le \epsilon\|f\|_2$.

The second step is a restricted isometry property for the measurement operator associated to random
sampling. Consider the sensing map acting on coefficient vectors,
$$
{\mathcal M}=P_XA^*:{\mathbb C}^D\to{\mathbb C}^D.
$$
The boundedness hypothesis
$$
\max_{x\in D}\max_{j\in D}|\phi_j(x)|\le M^{-\frac{1}{2}}
$$
is exactly the bounded orthonormal system hypothesis in compressive sensing. A standard result (see e.g. \cite[Chapter 12]{FoucartRauhut}) states that
if $|X|$ is at least a constant multiple of $|S|\log(|S|)^2\log M$, then with probability at least
$1-\exp(-c|X|)$ the matrix ${\mathcal M}$ satisfies a restricted isometry property of order $|S|$.

Since $|X|$ is binomial with mean $pM$, the hypothesis
$$
pM \ge C \frac{r^2}{\epsilon^2}\log\Bigl(\frac{r}{\epsilon}\Bigr)^2 \log M
$$
ensures, after adjusting constants, that $|X|$ is large enough for the restricted isometry property to hold
with probability at least $1-\exp(-c pM)$. One may justify the reduction from $pM$ to $|X|$ by a standard tail
estimate for sums of independent bounded random variables, for example Hoeffding's inequality \cite{Hoeffding63}.

The final step is stability of $\ell^1$ minimization for approximately sparse vectors under restricted isometry.
Once ${\mathcal M}$ has the restricted isometry property of order $|S|$, the standard compressive sensing recovery
theorem implies that the solution $f^*$ of
$$
\min_g \|Ag\|_1
\quad\text{subject to}\quad
\|P_X(g-f)\|_2\le \epsilon\|f\|_2
$$
satisfies
$$
\|f^*-f\|_2\le C_1\epsilon\|f\|_2,
$$
with an absolute constant $C_1$. This is proved in many places, including \cite[Chapters 4 and 6]{FoucartRauhut}.
In \cite{FRdiscrete} the constants are tracked through the same argument and the explicit numerical value
$C_1=11.47$ is recorded. Since the present proof follows the same chain of implications, we may take the same
constant here. This completes the proof.

\subsection{Proof of Theorem \ref{thm:obstruction}}

\begin{proof}
Let $\widehat f$ denote the two-dimensional Fourier transform of $f$ on
${\mathbb Z}_N\times{\mathbb Z}_T$, normalized so that Parseval holds. For each
$a\in{\mathbb Z}_T$, let $f_a(t)=f(t,a)$ and let $\widehat{f_a}$ denote the
one-dimensional Fourier transform of $f_a$ on ${\mathbb Z}_N$, again with Parseval
normalization. With these conventions, the two-dimensional transform may be viewed
row-wise and we have
$$
\widehat f(m,a)=\widehat{f_a}(m).
$$

We compare $\ell^1$ and $\ell^2$ norms across the $a$ variable. First,
$$
\|\widehat f\|_1
=
\sum_{a\in{\mathbb Z}_T}\sum_{m\in{\mathbb Z}_N}|\widehat f(m,a)|
=
\sum_{a\in{\mathbb Z}_T}\sum_{m\in{\mathbb Z}_N}|\widehat{f_a}(m)|
=
\sum_{a\in{\mathbb Z}_T}\|\widehat{f_a}\|_1.
$$
Applying Cauchy Schwarz in the index $a$ gives
$$
\|\widehat f\|_1
\le
\sqrt{T}\left(\sum_{a\in{\mathbb Z}_T}\|\widehat{f_a}\|_1^2\right)^{\frac12}.
$$

Next, Parseval for the two-dimensional transform and the orthogonality of the
row decomposition imply
$$
\|\widehat f\|_2^2
=
\sum_{a\in{\mathbb Z}_T}\sum_{m\in{\mathbb Z}_N}|\widehat f(m,a)|^2
=
\sum_{a\in{\mathbb Z}_T}\sum_{m\in{\mathbb Z}_N}|\widehat{f_a}(m)|^2
=
\sum_{a\in{\mathbb Z}_T}\|\widehat{f_a}\|_2^2.
$$
Therefore
$$
\|\widehat f\|_2
=
\left(\sum_{a\in{\mathbb Z}_T}\|\widehat{f_a}\|_2^2\right)^{\frac12}.
$$

Now set
$$
R=\max_{a\in{\mathbb Z}_T} FR(f_a)
=
\max_{a\in{\mathbb Z}_T}\frac{\|\widehat{f_a}\|_1}{\|\widehat{f_a}\|_2}.
$$
Then for every $a$ we have $\|\widehat{f_a}\|_1\le R\,\|\widehat{f_a}\|_2$, hence
$$
\left(\sum_{a\in{\mathbb Z}_T}\|\widehat{f_a}\|_1^2\right)^{\frac12}
\le
R\left(\sum_{a\in{\mathbb Z}_T}\|\widehat{f_a}\|_2^2\right)^{\frac12}
=
R\,\|\widehat f\|_2.
$$
Substituting into the earlier bound for $\|\widehat f\|_1$ yields
$$
\|\widehat f\|_1
\le
\sqrt{T}\,R\,\|\widehat f\|_2.
$$
Dividing by $\|\widehat f\|_2$ gives
$$
FR(f)=\frac{\|\widehat f\|_1}{\|\widehat f\|_2}
\le
\sqrt{T}\,\max_{a\in{\mathbb Z}_T} FR(f_a).
$$
Rearranging completes the proof, i.e.
$$
\max_{a\in{\mathbb Z}_T} FR(f_a)\ge \frac{FR(f)}{\sqrt{T}}.
$$
\end{proof}

\subsection{Proof of Theorem \ref{thm:kolmogorov-FR-abelian}}

The first ingredient is the same soft sparsification mechanism used in the proof of the global
Fourier-ratio recovery theorem. A refinement of Lemma \ref{lem:FR-sparse} using Theorem 4.5 in \cite{HavivRegev} yields the following statement: there exists a subset $S\subset\widehat G$ with
$$
|S|\le C_2\frac{r^2}{\epsilon^2}\log\Bigl(\frac{r}{\epsilon}\Bigr)^2 \log|G|
$$
such that if $\widehat f_S$ denotes the truncation of $\widehat f$ to $S$ and $f_S$ is the inverse
Fourier transform of $\widehat f_S$, then
$$
\|f-f_S\|_2\le \frac{\epsilon}{4}\|f\|_2.
$$
This is exactly the same approximation step that underlies \cite{FRdiscrete}, and it uses only
Parseval and the hypothesis $\|\widehat f\|_1\le r\|\widehat f\|_2$.

To turn this approximation into a finite description, we quantize the Fourier coefficients on $S$.
Let $k=|S|$ and enumerate $S=\{\gamma_1,\dots,\gamma_k\}$. For each $j$, write
$\widehat f(\gamma_j)=a_j+ib_j$ with $a_j,b_j\in{\mathbb R}$. Choose a quantization step size
$$
\delta=\frac{\epsilon}{4k^{1/2}}\|\widehat f\|_2.
$$
Let $\widetilde a_j$ and $\widetilde b_j$ be real numbers obtained by rounding $a_j$ and $b_j$ to a
nearest multiple of $\delta$, and define the quantized coefficient vector by
$$
\widetilde{\widehat f}(\gamma)=
\begin{cases}
\widetilde a_j+i\widetilde b_j & \text{if}\ \gamma=\gamma_j\ \text{for some}\ j,\\
0 & \text{if}\ \gamma\notin S.
\end{cases}
$$
Let $\widetilde f$ be the inverse Fourier transform of $\widetilde{\widehat f}$.

By construction,
$$
\|\widehat f_S-\widetilde{\widehat f}\|_2^2
=
\sum_{j=1}^k |(a_j-\widetilde a_j)+i(b_j-\widetilde b_j)|^2
\le
\sum_{j=1}^k (2\delta)^2
=
4k\delta^2,
$$
so
$$
\|\widehat f_S-\widetilde{\widehat f}\|_2\le 2k^{1/2}\delta=\frac{\epsilon}{2}\|\widehat f\|_2.
$$
By Parseval, $\|f_S-\widetilde f\|_2=\|\widehat f_S-\widetilde{\widehat f}\|_2$, hence
$$
\|f-\widetilde f\|_2
\le
\|f-f_S\|_2+\|f_S-\widetilde f\|_2
\le
\frac{\epsilon}{4}\|f\|_2+\frac{\epsilon}{2}\|f\|_2
\le
\epsilon\|f\|_2.
$$
Thus it suffices to upper bound the number of bits needed to describe $\widetilde f$.

A self-delimiting description of $\widetilde f$ consists of the following data together with a
fixed decoder program, whose length contributes an additive constant $C_U$.
First, list the characters $\gamma_1,\dots,\gamma_k$, which costs at most $C_3 k\log|G|$ bits by the
encoding convention for $\widehat G$. Second, record the quantized real and imaginary parts
$\widetilde a_j,\widetilde b_j$. Since each is an integer multiple of $\delta$, it suffices to record
the corresponding integers. The magnitudes satisfy
$$
|\widetilde a_j|+|\widetilde b_j|\le 2|\widehat f(\gamma_j)|+2\delta,
$$
so the required integer ranges are controlled by $\|\widehat f\|_2$ and $\delta$. This yields an
encoding cost bounded by
$$
C_4 k \log\Bigl(\frac{r}{\epsilon}\Bigr)
+
C_5 k,
$$
for absolute constants $C_4,C_5>0$.

Combining these estimates, we obtain
$$
K_U(f,\epsilon)
\le
C_6 k\log|G|+C_7 k\log\Bigl(\frac{r}{\epsilon}\Bigr)+C_U,
$$
and substituting
$$
k\le C_2\frac{r^2}{\epsilon^2}\log\Bigl(\frac{r}{\epsilon}\Bigr)^2
$$
gives the stated bound.

\subsection{Proof of Theorem \ref{thm:gaborstatdim}}

For $f \in B(r)$ we let $g = f_{\mathcal{B}}$. Note that
$$
\|g\|_2 = \|f\|_2 = \sqrt{M} \Rightarrow FR_{\mathcal{B}}(f) = \frac{\|g\|_1}{\sqrt{M}}.
$$

Note that for $f \in \mathbb{C}^G$,
$$
f(x) = \sum_{m \in G} \psi_{x,m} \cdot g(m)
$$
where $\{\psi_{x,m}\}$ represents the transformation from $\mathcal{B}$ to the standard basis on $G$. So, we let $Y$ be a random variable on $G$ such that
$$
Pr(Y = m) = \frac{|g(m)|}{\|g\|_1}
$$
and define a random function
$$
Z(x) = \|g\|_1 \psi_{x,Y} \frac{g(Y)}{|g(Y)|}
$$
so that for all $x \in G$, $E[Z(x)] = f(x).$ If we take $Z_1, ..., Z_k$ IID random functions with this distribution we get can use their mean as an estimate of $f$ and compute their mean-squared error. For any distribution, $D$, on $G$ with probability mass function $p$ we get
\begin{align*}
    E_{Z_1, ..., Z_k}\left[ E_D\left[ \left( f(x) - \frac{1}{k} \sum_{i = 1}^{k} Z_i(x) \right)^2 \right] \right]
    & = \sum_{x \in G} p(x) E_{Z_1, ..., Z_k}\left[ \left( f(x) - \frac{1}{k} \sum_{i = 1}^{k} Z_i(x) \right)^2 \right] \\
    & =  \sum_{x \in G} p(x) Var\left[ \frac{1}{k} \sum_{i = 1}^{k} Z_i(x) \right] \\
    & =  \sum_{x \in G} \frac{p(x)}{k} Var\left[ Z(x) \right]
\end{align*}

The modulus of this is bounded by
\begin{align*}
    \sum_{x \in G} \frac{p(x)}{k} E_p\left[ |Z(x)|^2 \right]
    & = \sum_{x \in G} \frac{p(x)}{k} \sum_{m} \frac{|g(m)|}{\|g\|_1} \|g\|_1^2 \cdot |\psi_{x, m}|^2 \\
    & \leq \frac{\tau^2}{k} \|g\|_1^2 \\
    & = \frac{M \tau^2 \cdot FR_{\mathcal{B}}(f)^2}{k} \\
    & \leq \frac{M \tau^2 r^2}{k}.
\end{align*}

For any $\varepsilon$, we can get the MSE $< \varepsilon^2$ by making $k \geq M \tau^2 r^2/\varepsilon^2$. This means that for any $f \in B(r)$ there is a degree $k$ ``polynomial'' of the form
$$
P(x) = \frac{\|g\|_1}{k} \sum_{i = 1}^{k} \psi_{x, m_i} \frac{g(m_i)}{|g(m_i)|}
$$
with $\|f - P\|_{L^2(D)} \leq \varepsilon$. We will now proceed to estimate these $P$ with a finite set of polynomials. Note that $0 \leq \|g\|_1 \leq r\sqrt{M}$ and the $g(m_i)/|g(m_i)|$ are on the unit circle. So, for $N_1, N_2 \in \mathbb{N}$ we we will approximate such $P$ with polynomials of the form
$$
\tilde{P}(x) = A \sum_{i = 1}^{k} \psi_{x, m_i} u_i
$$
where
$$
A \in \left\{ 0, \frac{r\sqrt{M}}{kN_1}, \frac{2r\sqrt{M}}{kN_1}, ..., \frac{r\sqrt{M}}{k} \right\}
$$
and the $u_i$ are the $N_2$ roots of unity. For any $P$, we take the $\tilde{P}$ with the same $m_i$ and where $A$ and the $u_i$ are as close as possible to to the corresponding coefficients in $P$. Since $\|P - \tilde{P}\|_{L_2(D)} \leq \|P - \tilde{P}\|_{\infty}$ we take arbitrary $x \in G$ and note that
\begin{align*}
    \left|P(x) - \tilde{P}(x)\right|
    &= \left| \frac{\|g\|_1}{k} \sum_{i = 1}^{k} \psi_{x, m_i} \frac{g(m_i)}{|g(m_i)|} - A \sum_{i = 1}^{k} \psi_{x, m_i} u_i\right| \\
    & \leq \sum_{i = 1}^{k} \left| \psi_{x, m_i} \right| \cdot \left| \frac{\|g\|_1}{k} \frac{g(m_i)}{|g(m_i)|} - A u_i\right| \\
    & \leq \sum_{i = 1}^{k} \left| \psi_{x, m_i} \right| \cdot \left( \left| \frac{\|g\|_1}{k} \right| \frac{1}{N_2} + \frac{1}{N_1} + \frac{1}{N_1 N_2} \right) \\
    & \leq \tau k \left( \left| \frac{r \sqrt{M}}{k} \right| \frac{1}{N_2} + \frac{1}{N_1} + \frac{1}{N_1 N_2} \right).
\end{align*}

So, by the triangle inequality we get that
$$
\left\| P - \tilde{P} \right\|_{L^2(D)}
\leq \varepsilon + \tau k \left( \left| \frac{r \sqrt{M}}{k} \right| \frac{1}{N_2} + \frac{1}{N_1} + \frac{1}{N_1 N_2} \right)
$$
which we denote $\rho$. The number of possible $\tilde{P}$ is bounded above by
$$
M^k (N_1 + 1) N_2^k
$$
which we denote $d$. This means that the covering number is bounded by
$$
\mathcal{N}\left(B(r), \| \cdot \|_{L^2(D)}, \rho\right) \leq d
$$
which means that the packing number is bounded by
$$
\mathcal{P}\left(B(r), \| \cdot \|_{L^2(D)}, 4\rho\right) \leq d.
$$

Thus, if we take $f_1, ..., f_d \in B(r)$ there is some $i \neq j$ such that
$$
\|f_i - f_j\|_{L^2(D)} \leq 4\rho
$$
$$
\therefore E_{x \sim D}[f_i(x) \cdot f_j(x)] \geq \frac{2 - 16\rho^2}{2}.
$$

Setting this $> 1/2 > 1/d$ we get that we need
$$
\rho = \varepsilon + \tau k \left( \left| \frac{r \sqrt{M}}{k} \right| \frac{1}{N_2} + \frac{1}{N_1} + \frac{1}{N_1 N_2} \right) \leq \frac{1}{4}.
$$

We can achieve this by setting each term $< 1/16$ which means we want
$$
\varepsilon \geq \frac{1}{16} \Rightarrow k \geq 16^2 M\tau^2 r^2
$$
$$
N_2\geq 16\tau r \sqrt{M}, \text{ and }
$$
$$
N_1 \geq 16\tau k.
$$

This, we can bound the statistical query dimension by
$$
d = M^{16^2 M\tau^2 r^2} (16^3\tau^3 M\ r^2 + 1) (16\tau r \sqrt{M})^{16^2 M\tau^2 r^2}.
$$

\section*{Acknowledgments}

The work of A.~Iosevich was supported in part by the National Science Foundation under Grant DMS--2506858.

\end{document}